\documentclass[reqno]{amsart}
\usepackage{amsmath}
\usepackage{monmacros}
\title[Dimension of monopoles on a.c.\ 3-mflds]{Dimension of monopoles on asymptotically conic 3-manifolds}
\author{Chris Kottke}
\address{Northeastern University\\Department of Mathematics}
\email{c.kottke@neu.edu}
\date{\today}
\begin{document}
\maketitle
\begin{abstract}
The virtual dimensions of both framed and unframed SU(2) magnetic monopoles on
asymptotically conic 3-manifolds are obtained by computing the index of a
Fredholm extension of the associated deformation complex. The unframed
dimension coincides with the one obtained by Braam for conformally compact
3-manifolds. The computation follows from the application of a Callias-type
index theorem.
\end{abstract}

\section{Introduction}

Magnetic monopoles have been studied in a variety of settings, going back to
the original work \cite{jaffe1980vortices,taubes1983stability,taubes1984monopoles} of Taubes, who
demonstrated that the moduli space $\frM_k(\bbR^3)$ of charge-$k$ monopoles on
$\bbR^3$ is a smooth, nonempty manifold of dimension $4k.$ Atiyah considered
the moduli space $\cM_k(\bbH^3)$ in \cite{atiyah1984magnetic}, and in
\cite{braam1989magnetic} Braam considered $\cM_k(X)$ for a general conformally
compact 3-manifold $X.$ In the posthumously published work
\cite{floer1995monopoles,floer1995configuration}, Floer outlined a construction
of monopoles on spaces with asymptotically Euclidean ends.

Here we consider an arbitrary asymptotically conic (a.k.a scattering)
3-manifold $(X,g)$, meaning $X$ is a manifold with boundary and $g$ has the
form $g = \tfrac{dx^2}{x^4} + \tfrac{h}{x^2}$, where $x$ is a boundary defining
function and $h$ restricts to a metric on $\pa X$. The usual definition of an
asymptotically conic manifold appearing in the literature, in terms of a radial
function $r$, is recovered by setting $x = 1/r$. Examples include the radial
compactification of $\R^3$, ALE spaces, and manifolds with Euclidean ends, as
well as manifolds with more general boundary surfaces. A monopole is a 
configuration $(A,\Phi)$ where $A$ is a connection on a fixed principal
$\SU(2)$-bundle $P \to X$ and $\Phi$ is a section of $\ad P$ satisfying the
Bogomolny equation
\begin{equation}
	\star F_A = d_A \Phi,
	\label{E:Bogo_intro}
\end{equation}
where $F_A$ is the curvature of $A$. Since the equation is gauge invariant, the
gauge group $\cG = \Gamma(X; \Ad P)$ acts on solutions, and the {\em charge $k$ monopole
moduli space}, $\cM_k(X)$, is the space of equivalence classes of solutions to
\eqref{E:Bogo_intro}, where $k \in \bbZ^{b^0(\pa X)}$ is a collection of
integers given by topological invariants of $\Phi$ over the components of $\pa
X$. Alternatively, one may consider the space of {\em framed monopoles},
where the boundary data $(A,\Phi)\rst_{\pa X}$ is fixed and equivalence is
taken with respect to the reduced gauge group $\cG_0$ which acts by the
identity at $\pa X$. This space is denoted $\frM_k(X).$ 

The {\em deformation complex} at a solution $(A,\Phi)$ is the elliptic complex
\begin{equation}
	T_1\cG \stackrel{\ssD_1}\to T_{(A,\Phi)} \cC_k \stackrel{\ssD_2}\to \Gamma(X; \Lambda^1\otimes \ad P)
	\label{E:def_complex_intro}
\end{equation}
where $\ssD_1$ is the infinitesimal action of the Lie algebra $T_1 \cG =
\Gamma(X; \ad P)$ of the gauge group, and $\ssD_2$ is the linearization of
\eqref{E:Bogo_intro} acting on the tangent space $T_{(A,\Phi)} \cC_k =
\Gamma(X; (\Lambda^1\oplus \Lambda^0)\otimes \ad P)$ to the configurations at
$(A,\Phi)$.  The tangent space to the moduli space, $T_{(A,\Phi)} \cM_k$, may
be formally identified with the middle degree cohomology of the deformation
complex, so in particular $\dim(\cM_k) = \dim(\kernel \ssD_2/\image \ssD_1)$,
while the {\em virtual dimension} is the Euler characteristic
\[
	\vdim(\cM_k) = \dim(\kernel \ssD_2/\image \ssD_1) - \dim\kernel \ssD_1 - \dim \cokernel \ssD_2.
\]
A similar deformation complex may be considered for framed monopoles, taking
$T_1 \cG_0$ to be sections of the gauge algebra which vanish at $\pa X$ and
$T_{(A,\Phi)}\cC_k$ to be perturbations fixing the boundary data. 

We define a family of completions of \eqref{E:def_complex_intro} as Hilbert complexes:
\begin{equation}
	\cH^{\gamma-1,2}(X; \ad P) \stackrel{\ssD_1}{\to} \cH^{\gamma,1}(X; (\Lambda^1\oplus \Lambda^0)\otimes \ad P) \stackrel{\ssD_2}{\to} \cH^{\gamma+1,0}(X; \Lambda^1\otimes \ad P)
	\label{E:hilb_complex_intro}
\end{equation}
where $\gamma \in \bbR$ is a real parameter. (These spaces are defined in
detail in \S\ref{S:fred_extn}; some notation is suppressed here.) These are
Sobolev spaces contained within weighted $L^2$ spaces:
\[
	\cH^{\gamma,l} \subset x^\gamma L^2,
\]
which for $\gamma \leq -\tfrac{1}{2}$ give Hilbert completions of the unframed
deformation complex and for $\gamma > -\tfrac {1}{2}$ give completions of the
framed complex. The main result of this paper is:

\begin{thm}
The complex \eqref{E:hilb_complex_intro} is Fredholm (i.e., has finite
dimensional cohomology) for $\gamma \in (-\tfrac{1}{2},-\tfrac{1}{2} +
\lambda_1)$ and for $\gamma \in (-\tfrac{3}{2},-\tfrac{1}{2})$, where $\lambda_1
= \sqrt{\nu_1 + \tfrac 1 4} - \tfrac 1 2$ and $\nu_1$ is the least positive
eigenvalue of $\Delta_{\pa X}.$ Furthermore
\[
\begin{aligned}
	\vdim\frM_k(X) 
	&= 4\ul k - \tfrac 1 2 b^1(\pa X),
	= \ind\big(\ssD_2(\gamma) + \ssD_1(\gamma)'\big) 
	&  \gamma &\in (-\tfrac 1 2,-\tfrac{1}{2} + \lambda_1),
	\\\vdim\cM_k(X) 
	&= 4\ul k + \tfrac 1 2 b^1(\pa X) - b^0(\pa X), 
	= \ind\big(\ssD_2(\gamma) + \ssD_1(\gamma)'\big) 
	& \gamma &\in (-\tfrac{3}{2},-\tfrac{1}{2}),
\end{aligned}
\]
where the total charge $\ul k = \sum_{i = 1}^{b^0(\pa X)} k_i$ is the sum of
the charges over the ends of $X$, and $b^i(\pa X)$ denotes the $i$th Betti
number $\pa X$.
\label{T:intro_thm}
\end{thm}
\noindent Several remarks are in order:
\begin{itemize}
\item 
Theorem~\ref{T:intro_thm} gives a new proof of the classical results
$\vdim\big(\cM_k(\bbR^3)\big) = 4k$ and $\vdim\big(\frM_k(\bbR^3)\big) = 4k -
1$, which are the true moduli dimensions in this case, since $\bbR^3$ is a
scattering manifold with one end, $\pa \bbR^3 = S^2$, for which $b^1(S^2) = 0$
and $b^0(S^2) = 1$. 
\item
The virtual dimensions may be re-expressed using the identity $\tfrac 1 2
b^1(\pa X) = b^1(X) - b^2(X) +(b^0(\pa X) - 1)$, which follows from duality and
the long exact sequence in cohomology of the pair $(X,\pa X)$.  The virtual
dimension $\vdim(\cM_k(X)) = 4\ul k + b^1(X) - b^2(X) - 1$ coincides with the
one obtained by Braam in \cite{braam1989magnetic} for conformally compact
manifolds, even though that setting is quite different from an analytical point
of view. 
\item 
The difference $\vdim(\cM_k(X)) - \vdim(\frM_k(X)) = b^1(\pa X) - b^0(\pa X)$
has a geometric interpretation in terms of the moduli space of monopole
boundary data and the action of the gauge group on such data. This is discussed
in \S\ref{S:framing} below.
\item
If $\nu_1 > \tfrac{3}{2}$, then the range of $\gamma$ for which the framed
deformation complex is Fredholm includes $\gamma = 0$, at which value the
infinitesimal perturbations of $(A,\Phi)$ are contained in $L^2$, and $T\frM_k$
inherits a Riemannian metric in terms of the $L^2$ pairing. In the classical
case $X = \bbR^3$ (for which $\nu_1(S^2) = 2$), this metric is famously known
to be complete and hyperk\"ahler \cite{atiyah1988geometry}.
\end{itemize} 

In \S\ref{S:def} we discuss the definition of monopoles on a scattering
manifold, consider the issues around framing and set up the deformation complex
along with the precise family of Hilbert completions of this complex that we
shall consider. The starting point for the proof of Theorem~\ref{T:intro_thm}
is a generalized Callias-type index theorem, Theorem~\ref{T:index_theorem}
below, which is proved in \cite{kottke2013callias}; we recall this result in
\S\ref{S:callias_review}. In \S\ref{S:comp} we apply this theorem to the Dirac
operators obtained from the Hilbert complex \eqref{E:def_complex_intro},
arriving at the result above.

The main analytical feature of our theory is this: the family of Hilbert
complexes leads to to a family of Sobolev extensions for the associated Dirac
operator $\ssD_2(\gamma) + \ssD_1(\gamma)'$, depending in particular on a
weight parameter $\gamma \in \bbR$. These extensions are Fredholm for $\gamma$
outside of a discrete set of {\em indicial roots} (which here have expressions
in terms of the eigenvalues of the Laplacian on $\pa X$), and the index of the
extension changes (by the dimension of the associated eigenspace) as $\gamma$
varies. This phenomenon of variable index Fredholm extensions on weighted
Sobolev spaces goes back to the work of Lockhart and McOwen
\cite{lockhart1995elliptic}, the `b-calculus' of Melrose
\cite{melrose1993atiyah}, as well as the work of Schulze et.\ al.\
\cite{schulze1998index}. More recently, it has appeared in a range of settings
including problems in scattering theory \cite{borthwick2001scattering},
\cite{guillarmou2008resolvent}, closed extensions of conic differential
operators \cite{gil2003adjoints}, and of operators on stratified spaces
\cite{albin2012signature}, among many others. Here the behavior of the operator
at infinity leads to the use of `hybrid' b-/scattering-type Sobolev spaces
adapted to a splitting of the vector bundle there. Similar hybrid Sobolev
spaces have also appeared in \cite{hausel2004hodge} and
\cite{grieser2014parametrix}. One novelty of the problem presented here is
that some of the indicial roots themselves depend on the parameter $\gamma$, so
that the index can both increase {\em and} decrease as $\gamma$ decreases (see
Figure~\ref{F:bspec} and the associated discussion in \S\ref{S:comp_index}).

Finally, we expect that the approach described here to the computation of
monopole moduli dimensions, via the application of the Callias-type index
theorem \cite{kottke2013callias} to the deformation complex, should generalize
to quite a few situations of interest. Among these we mention monopoles with
higher rank gauge groups on $\R^3$ \cite{murray2003note} as well as more
general asymptotically conic manifolds, and monopoles on higher dimensional
manifolds with special holonomy; Oliveira in \cite{oliveira2013monopoles} has
recently obtained some results regarding monopoles on Bryant-Salamon G2
manifolds.
 
\subsection*{Acknowledgments} The present paper represents an extension of the
author's PhD thesis work, and he is grateful to his thesis advisor Richard
Melrose for his support and guidance. The author has also benefited from
numerous helpful discussions with Michael Singer, Gon\c calo Oliveira and
Pierre Albin, and would like to thank Kiril Datchev for his comments on the
manuscript.

\section{Monopoles and deformation} \label{S:def}
Let $(X,g)$ be a 3-dimensional manifold with boundary with $g$ a Riemannian
scattering metric on the interior of $X$. By a result of Joshi and Sa Baretto
\cite{joshi1999recovering}, we may assume $g$ is an {\em exact scattering
metric}, i.e.\ of the form
\begin{equation}
	g = \frac{dx^2}{x^4} + \frac{h}{x^2}
	\label{E:ex_sc_metric}
\end{equation}
near $\pa X$ with respect to a fixed boundary defining function $x,$ where $h$
is a bounded family of metrics on $\pa X$.  

Fix a principal $\SU(2)$ bundle $P \to X$. ($P$ is necessarily trivializable
since $\SU(2)$ is $2$-connected, though we do not fix a trivialization.) The
{\em configuration space}, $\cC(X)$, for magnetic monopoles consists of pairs
$(A,\Phi)$ where $A$ is a connection on $P$ and $\Phi$ is a section of $\ad P =
P\times_{\ad} \su(2)$, a bundle which we equip with a Hermitian inner product
given by the negative of the  Killing form.  It is unreasonable to expect
monopoles on a general $X$ to be smooth, so we consider configurations which
are {\em bounded polyhomogeneous}, meaning they are smooth on the interior of
$X$, continuous up to the boundary, and have complete asymptotic expansions at
$\pa X$ in real powers of $x$ and non-negative integer powers of $\log x$ (see
\cite{kottke2013callias} for a more detailed discussion). Thus
\[
	\cC(X) = \cA(P)\times \Gamma(X; \ad P),
\]
where $\cA(P)$ is an affine space modelled on $\Gamma(X; T^*X\otimes \ad P)$
and we use the notation $\Gamma(X;V)$ to denote bounded polyhomogeneous
sections of a vector bundle $V.$ 

The configuration space is acted on by the {\em gauge group} $\cG(X) = \Gamma(X;
\Ad\, P)$, and a magnetic monopole is a gauge equivalence class of solutions to
the {\em Bogomolny equation}
\begin{equation}
	\cB(A,\Phi) = \star F_A - d_A \Phi = 0.
	\label{E:bogomolny}
\end{equation}
where $F_A$ is the curvature of $A$ and $d_A$ is the covariant derivative
defined by $A.$ Monopoles are minimizers of the Yang-Mills-Higgs action
\begin{equation}
	(A,\Phi) \mapsto \norm{F_A}^2_{L^2} + \norm{d_A \Phi}^2_{L^2},
	\label{E:action}
\end{equation}
and the part of $\cC(X)$ on which the action is finite decomposes into
connected components $\cC_k(X)$ indexed by an integral parameter $k \in
\bbZ^{b^0(\pa X)}$ known as the {\em charge}. Indeed, since $X$ is complete,
finite action implies that $(d_A \Phi)\rst_{\pa X}$ vanishes, so $\abs{\Phi}
\rst_{\pa X} = m$ is a constant known as the {\em mass} which we assume is
strictly positive and fix throughout, and then the charge is defined by
\begin{equation}
\begin{gathered}
	k = c_1(L) \in H^2(\pa X; \Z) \cong \Z^{b^0(\pa X)}, \\
	\quad \Phi\rst_{\pa X} \cong \begin{pmatrix} im & 0\\0 & -im\end{pmatrix} \in \Gamma\bpns{\pa X; \End(L\oplus L^{-1})}.
\end{gathered}
	\label{E:eigenbundle_L}
\end{equation}
Here $L$ is the line bundle spanned by the positive imaginary eigenvectors of
$\Phi \rst_{\pa X}$ on the $\C^2$ bundle over $\pa X$ associated to the
standard representation of $\SU(2)$---in other words, viewing $\Phi \rst_{\pa
X}$ as a skew-adjoint $2\times2$ matrix. 

We denote the (unframed) moduli space of charge $k$ monopoles by
\begin{equation}
\begin{gathered}
	\cM_k(X) = \cB^{-1}(0) / \cG(X), 
	\\ \cB : \cC_k(X) \to \Gamma(X; T^\ast X\otimes \ad P).
\end{gathered}
	\label{E:unframed_moduli}
\end{equation}
(The space also depends on the mass $m$, but we suppress this from the
notation.) We will also consider {\em framed monopoles}, wherein the boundary
data $(A_0,\Phi_0) = (A,\Phi)\rst_{\pa X}$ is fixed and the gauge group is
restricted to the subgroup $\cG_0(X) = \set{g \in \cG : g \rst_{\pa X} = 1}$.
We denote the framed moduli space of charge $k$ monopoles by
\begin{equation}
\begin{gathered}
	\frM_k(X) = \frM_k(X; A_0,\Phi_0) = \cB^{-1}(0)/\cG_0(X), 
	\\ \cB : \cC_k(X; A_0,\Phi_0) \to \Gamma(X; T^\ast X\otimes \ad P), 
	\\ \cC_k(X; A_0,\Phi_0) = \set{(A,\phi) \in \cC_k(X) : (A,\phi)\rst_{\pa X} = (A_0,\Phi_0)}.
\end{gathered}
	\label{E:framed_moduli}
\end{equation}

\subsection{Framing and monopole boundary data} \label{S:framing}

To appreciate the relative dimension of $\cM_k(X)$ versus $\frM_k(X)$, some
further discussion of framing is in order. First of all, the conditions $\abs
\Phi \rst_{\pa X} = m$ and $d_A \Phi \rst_{\pa X} = 0$ imply that the bundle
$P$ and connection $A$ admit a reduction over $\pa X$ to a principal bundle $Q
\to \pa X$ with structure group $\UU(1)$, the stabilizer of $\Phi \rst_{\pa
X};$ this is nothing more than the frame bundle of the line bundle $L \to \pa X$
described above. 

In fact more can be said. Consider the expansion
\[
	\Phi \sim \Phi_0 + \cdots + \Phi_1 x + \cO(x^{1 + \epsilon})
\]
with respect to the fixed boundary defining function $x$ (we ignore any
asymptotics between $x^0$ and $x^1$, since the coupling in the Bogomolny
equation occurs only between coefficients with integer offsets). Imposing the
Bogomolny equation formally implies:
\[
\begin{gathered}
	d_{A_0} \Phi_0 = 0, \quad [\Phi_0,\Phi_1] = 0, \\
	d_{A_0} \Phi_1 = 0, \quad \Phi_1 = \star_{\pa X} F_{A_0},
\end{gathered}
\]
where $A_0$ denotes the restriction of $A$ to $\pa X$ and $F_{A_0}$ denotes its
curvature. It follows from $d_{A_0} {\star F_{A_0}} = 0$ that there exists a
reduction $(Q,A_0)$ of $(P,A_0) \rst_{\pa X}$ to a $\UU(1)$-bundle with
connection such that $\star F_{A_0}$ and $\Phi_0$ are {\em constant} (c.f.\
\cite{atiyah1983yang}, proof of Theorem~6.7). Fixing such a reduction reduces
the gauge group $\cG(X)$ to the subgroup having boundary values in the $\UU(1)$
gauge group $C^\infty(\pa X; \Ad\, Q).$

Thus the space of charge $k$ monopole boundary data can be regarded as the
space of connections on the degree $k$ $\UU(1)$-bundle $Q_k \to \pa X$ (which
is unique up to isomorphism) with prescribed constant curvature (meaning a
constant multiple of the volume form), up to gauge. If $b^1(\pa X) = 0$, all
such connections are gauge equivalent, so the space of monopole boundary data
is discrete. However if $b^1(\pa X) \neq 0$, one can alter $(Q,A_0)$ by
tensoring with a flat connection (the space of which is the torus $H^1(\pa X;
\UU(1))$) and these are generally gauge inequivalent. Thus, denoting the moduli
space of monopole boundary data by $(\pa \cM)_k(X)$, we expect in general that
\begin{equation}
	\dim (\pa \cM)_k = b^1(\pa X).
	\label{E:dim_mon_boundary_data}
\end{equation}
Restriction defines a map $R : \cM_k(X) \to (\pa \cM)_k(X)$, and
\eqref{E:dim_mon_boundary_data} accounts for part of the expected difference in
dimension between $\cM_k(X)$ and $\frM_k(X)$. However, $\frM_k(X)$ is {\em not}
simply given by $R^{-1}([A_0,\Phi_0])$; there is an additional contribution
coming from the gauge group. 

Recall that in the classical case $\frM_k(\bbR^3) \to \cM_k(\bbR^3)$ is a
circle bundle; the extra dimension is accounted for by the fact that there is
an explicit one-parameter subgroup of $\cG(\bbR^3)$, namely $\set{\exp (\lambda
\Phi) : \lambda \in \bbR}$, which acts freely on $\cC_k(\bbR^3)$,  $k \neq 0$,
but which fixes the boundary data and yet does not lie in $\cG_0(X)$ (see for
instance \cite{atiyah1988geometry}). This may be generalized to the present
case, in which there is a $b^0(\pa X)$-dimensional subgroup acting freely but
fixing the boundary data; it is generated by
\[
	\bbR \ni \lambda \mapsto \exp(\lambda \chi_i \Phi), \quad i = 1, \ldots, b^0(\pa X),
\]
where $\chi_i$ is a smooth cutoff near the $i$th component of $\pa X.$ That
these gauge transformations act non-trivially if $k \neq 0$ while fixing the
boundary data can be seen from the infinitesimal action \eqref{E:inf_gauge}
below. This subgroup acts on framed monopole configurations, and yet two 
configurations differing by such a transformation are not regarded as equivalent,
since the quotient in \eqref{E:framed_moduli} is by $\cG_0(X)$, which does not
contain the subgroup in question.

In light of these two considerations it is reasonable to expect that
\[
	\dim\cM_k(X) - \dim\frM_k(X) = b^1(\pa X) - b^0(\pa X)
\]
in general. Though this equation is merely heuristic at this point, it is borne
out by the analysis.

\subsection{Deformation complex} \label{S:def_cplx}

The problem of computing the formal dimension of $\cM_k(X)$ or $\frM_k(X)$ is
an infinitesimal one, and may be recast in the form of an elliptic complex. We
proceed to define the deformation complex formally at first, before completing
to a Hilbert complex. In what follows, we will use the {\em scattering
cotangent bundle} $\scT^\ast X$, a rescaled cotangent bundle with respect to
which the metric \eqref{E:ex_sc_metric} is Hermitian and nondegenerate up to
the boundary (see \cite{kottke2013callias} or \cite{melrose1994spectral} for
more details). There is a natural map $T^\ast X \to \scT^\ast X$, and we will
use the shorthand $\Lambda^k$ to denote the bundle $\Lambda^k (\scT^\ast X).$

At a pair $(A,\Phi)$, the tangent space to the configuration
space is $T_{(A,\Phi)} \cC = \Gamma(X; \Lambda^1 \otimes \ad P) \oplus \Gamma(X; \ad P)$, while
the Lie algebra of the gauge group is $T_1 \cG = \Gamma(X; \ad P)$. The derivative
of the gauge action at $(A,\Phi)$ gives a map
\begin{equation}
	T_1 \cG \ni \gamma \mapsto (-d_A \gamma, -\ad\Phi(\gamma)) \in T_{(A,\Phi)} \cC,
	\label{E:inf_gauge}
\end{equation}
where $\ad \Phi = [\Phi,\cdot] \in \Gamma(X; \End(\ad P))$.  On the other hand,
linearizing the Bogomolny equation \eqref{E:bogomolny} defines a map
\begin{equation}
	d\cB : T_{(A,\Phi)} \cC \ni (a,\phi) \mapsto \star d_A a - d_A \phi + \ad\Phi(a) \in \Gamma(X; \Lambda^1\otimes \ad P).
	\label{E:lin_bogo}
\end{equation}
It is convenient at this point to make use of the isomorphism $\star : \Gamma(X;
\ad P) \cong \Gamma(X; \Lambda^3\otimes \ad P)$, after which we may 
arrange \eqref{E:inf_gauge} and \eqref{E:lin_bogo} into a sequence
\begin{equation}
\begin{gathered}
	\Gamma(X; \Lambda^3\otimes \ad P) \stackrel {\ssD_1} \to 
	  \Gamma(X; \Lambda^1\otimes \ad P)\oplus \Gamma(X; \Lambda^3\otimes \ad P) \stackrel {\ssD_2} \to
	  \Gamma(X; \Lambda^1\otimes \ad P), \\
	\ssD_1 : \gamma \mapsto (-d_A {\star \gamma}, -\ad\Phi(\gamma)), 
	  \quad \ssD_2 : (a,\phi) \mapsto \star d_A a -  d_A {\star \phi} + \ad\Phi(a),
	\label{E:def_complex}
\end{gathered}
\end{equation}
where $\ssD_1$ represents the infinitesimal gauge group action and $\ssD_2$
represents the linearization of the Bogomolny equation. The condition
\[
	\ssD_1^\ast (a,\star \phi) = 0 \iff -\delta_A a + \ad\Phi(\phi) = 0
\]
is known classically as the {\em Coulomb gauge condition,} where $\delta_A =
(d_A)^\ast = (-1)^k {\star d_A} \star$ is the formal adjoint of $d_A$ on forms
of degree $k$, with respect to the $L^2$ pairing determined by the metric and
inner product on $\ad P.$ For later reference we observe that $\ad\Phi$ is a
skew-adjoint endomorphism of $\ad P$ and $\star^\ast = \star^{-1} = \star$
since the dimension of $X$ is odd.

\begin{prop}
If $(A,\Phi)$ satisfies the Bogomolny equation \eqref{E:bogomolny}, then the
sequence \eqref{E:def_complex} is an elliptic chain complex.
\label{P:def_complex}
\end{prop}
\begin{proof}
Indeed,
\[
\begin{aligned}
	\ssD_2\,\ssD_1 \gamma &= - \star d_A d_A {\star \gamma} + d_A (\ad\Phi(\star \gamma)) - \ad\Phi(d_A {\star \gamma})
	\\&= - [\star F_A,\star \gamma] + [d_A \Phi, \star \gamma] + [\Phi,d_A {\star \gamma}] - [\Phi,d_A {\star \gamma}]
	\\&= [-{\star F_A} + d_A \Phi,\star \gamma]
\end{aligned}
\]
which vanishes if $\star F_A = d_A \Phi.$ 
At the principal symbolic level, 
\[
\begin{gathered}
	\sigma(\ssD_1)(\xi)(v_3) = (- i \xi \smwedge  {\star v_3}, 0) = (-{\star i}\xi \smiprod v_3, 0), \quad \text{and}
	\\ \sigma(\ssD_2)(\xi)(w_1,w_3) = \star i \xi \smwedge w_1 - i \xi \smwedge {\star w_3}
	= \star (i \xi \smwedge w_1 - i \xi \smiprod w_3).
\end{gathered}
\]
These determine an exact complex, since $\kernel(\sigma(\ssD_2)(\xi)) =
\set{(w_1,0) : w_1 = \xi \otimes a}$ lies in the image of
$\sigma(\ssD_1)(\xi).$
\end{proof}

From now on we assume that $(A,\Phi)$ satisfies
\eqref{E:bogomolny}. Formally speaking, the tangent space of $\cM_k$ at
$(A,\Phi)$ is represented by the the degree 1 cohomology space of
\eqref{E:def_complex}:
\[
	T_{(A,\Phi)}\cM_k = \sH^1 = \pns{\kernel \ssD_2/\image \ssD_1},
\]
and $\dim(\sH^1)$ computes the dimension of $\cM_k$ assuming it is smooth at
$(A,\Phi).$ On the other hand, the {\em virtual dimension} of $\cM_k$ is the 
Euler characteristic
\[
	\vdim(\cM_k) = \dim \sH^1 - \pns{\dim \sH^0 +  \dim \sH^2}
\]
which gives the true dimension of $\cM_k$ if $\sH^0 = \sH^2 = \set{0}$---in
other words, if $\ssD_1$ is injective, meaning the gauge group acts freely at
$(A,\Phi)$, and $\ssD_2$ is surjective, so that $(A,\Phi)$ is a regular point
of $\cB.$ 

\subsection{Fredholm extension} \label{S:fred_extn}

We proceed to compute the virtual dimension by Hodge theoretic methods, as the
index of $\ssD_2 + \ssD_1'$ with respect to a suitable Fredholm extension. We
first define weighted $L^2$ spaces with respect to which \eqref{E:def_complex}
becomes a complex of unbounded operators on Hilbert spaces; for technical
reasons encountered below we need to consider different weights along different
directions in $\ad P$ at infinity. 

To this end, consider a collar neighborhood $U \cong \pa X \times[0,\epsilon)$ of
$\pa X$ in which $\Phi \neq 0$ and set
\begin{equation}
\begin{gathered}
	\ad P \rst_U = \ad P_0 \oplus \ad P_+\oplus \ad P_-,
	\\ \ad P_0 := \bbC \Phi, 
	\quad \ad P_1 = \ad P_+ \oplus \ad P_- := \Phi^\perp.
	\label{E:adP_decomp}
\end{gathered}
\end{equation}
Thus $\ad P_0$ is the kernel of $\ad\Phi$, which is nondegenerate on $\ad P_1$,
and the later further splits into positive/negative imaginary eigenspaces $\ad
P_{\pm}$ of $\ad \Phi.$ In fact, by simplicity of $\su(2)$, we may take $\Phi$
to be proportional to the Cartan element at each point, and then the orthogonal
decomposition \eqref{E:adP_decomp} coincides with the root space decomposition
$\su(2)_\bbC \cong \slfrak(2,\bbC) = \mathfrak{h}\oplus \mathfrak{g}_\alpha
\oplus \mathfrak{g}_{-\alpha}$.  For later reference, we record the
relationship between these bundles and the line bundle $L$ defining the charge
in \eqref{E:eigenbundle_L} in the following result, which follows easily by
decomposing into irreducible representations of $\su(2).$

\begin{lem}
Over $\pa X$, the complex line bundles $\ad P_+$ and $L\otimes L$ (respectively
$\ad P_-$ and $L^\ast \otimes L^\ast$) are isomorphic. Thus,
\[
	\ad P \rst_{\pa X} = \ad P_0 \oplus \ad P_+\oplus \ad P_- \cong \underline{\C}\oplus L^{2}\oplus L^{-2}
\]
where $\underline \C$ denotes the trivial bundle.
\label{L:monopole_bundle_isoms}
\end{lem}
%

Let $\Pi_0$ denote the projection onto $\ad P_0$ over $U$ and $\chi \in
C^\infty_c(U;[0,1])$ a smooth cutoff with $\chi \equiv 1$ near $\pa X$. Then,
for $\alpha,\beta \in \bbR$, define the space $\cL^{\alpha,\beta}(X; \ad
P\otimes \Lambda^\ast)$ to be the completion of $C^\infty_c(\mathring X; \ad P\otimes
\Lambda^\ast)$ with respect to the norm
\[
\begin{gathered}
	\norm{u}^2_{\cL^{\alpha,\beta}} = \norm{x^{-\alpha} u_0}_{L^2}^2 
	+ \norm{x^{-\beta} u_1}_{L^2}^2
	+ \norm{u_c}_{L^2}^2,
	\\ u = u_0 + u_1 + u_c := \Pi_0 (\chi u) + (\id - \Pi_0)(\chi u) + (1 - \chi)u.
\end{gathered}
\]
In other words, near the boundary, 
\[
	\cL^{\alpha,\beta}(X; \ad P\otimes \Lambda^\ast) 
	  \simeq x^{\alpha} L^2(U; \ad P_0\otimes \Lambda^\ast)
	  \oplus x^{\beta} L^2(U; \ad P_1\otimes \Lambda^\ast), 
	\quad \text{over $U$.}
\]
These are Hilbert spaces, with inner product obtained by polarization.

Applying this to \eqref{E:def_complex}, we consider the family of 
unbounded elliptic complexes parameterized by $\gamma \in \bbR$:
\begin{multline}
	\cL^{\gamma-1,\gamma+1}(X; \ad P\otimes \Lambda^3) 
	\stackrel{\ssD_1}{\to} \cL^{\gamma,\gamma+1}\big(X; \ad P\otimes (\Lambda^1\oplus \Lambda^3)\big) 
	\\ \stackrel{\ssD_2}{\to} \cL^{\gamma+1,\gamma+1}(X; \ad P\otimes \Lambda^1).
	\label{E:L2_complex}
\end{multline}
These particular choices of weights are necessitated by the index theorem
applied below. To motivate the increase in weight along $\ad P_0$ at each step,
note that on $\ad P_0 = \bbC \Phi$ the term $\ad\Phi$ vanishes, so the
operators $\ssD_i$, $i = 1,2$ each have the form $\pm {\star d_A}$ or $d_A
\star$, from which a power of $x$ may be factored out. This is discussed in
more detail below.  

It remains to specify domains for $\ssD_1$ and $\ssD_2$ in
\eqref{E:L2_complex}. Following the analysis in \cite{kottke2013callias}, we
define Sobolev spaces $\sH^{\alpha,\beta,k,l}(X; \ad P\otimes \Lambda^*)$, where
$\alpha,\beta \in \bbR$, $k,l \in \bbN_0$, as the completions of $C^\infty_c(\mathring X;
\ad P\otimes \Lambda^*)$ with respect to the norms
\[
	\norm{u}_{\cH^{\alpha,\beta,k,l}}^2 =  \norm{x^{-\alpha} (x^{-1}\nabla)^k(\nabla)^l u_0}^2_{L^2} + 
	\norm{x^{-\beta} (\nabla)^{k+l} u_1}^2_{L^2} + \norm{(\nabla)^{k+l} u_c}^2_{L^2}.
\]
In particular, regularity is measured differently near $\pa X$ along $\ad P_0$
compared to $\ad P_1$, in that $k$ of the $k+l$ derivatives along the $\ad P_0$
are weighted by $x^{-1}$; on Euclidean space this corresponds to using the
radially weighted derivatives $r\pa_r$ and $\pa_\theta$ rather than $\pa_r$ and
$r^{-1}\pa_\theta$. 

We finally arrive at the object of primary consideration---the family of
complexes parameterized by $\gamma \in \bbR$, $k \in \bbN$:
\begin{multline}
	\cH^{\gamma-1,\gamma+1,k,2}(X; \ad P\otimes \Lambda^3) 
	\stackrel{\ssD_1}{\to} \cH^{\gamma,\gamma+1,k,1}\big(X; \ad P\otimes (\Lambda^1\oplus \Lambda^3)\big) 
	\\ \stackrel{\ssD_2}{\to} \cH^{\gamma+1,\gamma+1,k,0}(X; \ad P\otimes \Lambda^1).
	\label{E:Hilb_complex}
\end{multline}
Considered as domains in \eqref{E:L2_complex}, these determine {\em Hilbert
complexes}, in the sense of \cite{bruning1992hilbert}. Below we determine the
values of $\gamma$ for which \eqref{E:Hilb_complex} is Fredholm and compute its
index.

Before doing so however, two remarks are in order. First, note that the cutoff
for bounded sections to be in $x^\alpha L^2$ on a scattering 3-manifold is
$\alpha = -\tfrac{3}2$; more precisely, for $\alpha \geq -\tfrac{3}{2}$ any continuous sections
in $x^\alpha L^2$ must vanish at $\pa X$ while for $\alpha < -\tfrac{3}2$ they may be
nonzero up to $\pa X.$ It follows that for $\gamma \geq -\tfrac{1}2$ the leftmost
space in \eqref{E:L2_complex} is a weighted $L^2$ completion of the {\em
reduced} gauge Lie algebra $T_1 \cG_0$, while for $\gamma < -\tfrac 1 2$ it represents
a weighted $L^2$ version of the {\em full} gauge Lie algebra $T_1 \cG$.\footnote{The
extra vanishing along $\ad P_1$ is required here only for technical reasons.
With a judicious choice of gauge for $(A,\Phi)$, the weights along $\ad P_0$
and $\ad P_1$ can be considered independently (see \cite{kottkegluing}), and the index
computed below does not depend on the chosen weight along $\ad P_1$.} Thus,
denoting by $\sH^\ast(\gamma)$ the cohomology spaces of \eqref{E:Hilb_complex},
for $\gamma$ sufficiently near $-\tfrac{1}2$,
\begin{equation}
	\dim \sH^1(\gamma) - (\dim \sH^0(\gamma) + \dim \sH^2(\gamma)) 
	= \begin{cases} \vdim(\cM_k), & \gamma \geq -\tfrac 1 2
	\\ \vdim(\frM_k), & \gamma < -\tfrac 1 2. \end{cases}
	\label{E:vdim_parameter}
\end{equation}

The second remark concerns the behavior of adjoints in the complex
\eqref{E:L2_complex}.  As a notational convention, we denote by $\ssD_1' =
\ssD_1(\gamma)'$ the adjoint of $\ssD_1$ as an operator \eqref{E:L2_complex},
and denote by $\ssD_1^*$ its formal $L^2$ adjoint (with which is it more
convenient to work).  As a result of the weights, these are related via
\begin{equation}
\begin{gathered}
	\ssD_1(\gamma)' = \rho(\gamma) \ssD_1^\ast \rho(\gamma)^{-1}
	= \ssD_1^\ast + [\ssD_1^\ast,\rho(\gamma)^{-1}]
	: \cL^{\gamma,\gamma+1}
	\to \cL^{\gamma-1,\gamma+1},
	\\ \rho(\gamma) = \begin{pmatrix} x^{2\gamma} & 0\\0 & x^{2(\gamma+1)}\end{pmatrix}
	\quad \text{with respect to $\ad P = \ad P_0\oplus \ad P_1$ near $\pa X$.}
\end{gathered}
	\label{E:weighted_adjoint}
\end{equation}

According to the theory of Hilbert complexes, the complex
\eqref{E:Hilb_complex} is {\em Fredholm}, i.e.\ has finite dimensional
cohomology spaces, if and only if the operator 
\[
	\ssD_2(\gamma) + \ssD_1(\gamma)' : \cH^{\gamma,\gamma+1,k,1}\big(X; \ad P\otimes \Lambda^\odd)
	\to \cH^{\gamma+1,\gamma+1,k,0}\big(X; \ad P\otimes \Lambda^\odd)
\]
is Fredholm, and then the index of the operator equals the Euler characteristic
\eqref{E:vdim_parameter}. From \eqref{E:def_complex} and
\eqref{E:weighted_adjoint}, we may write
\begin{equation}
	\ssD_2(\gamma) + \ssD_1(\gamma)'
	= \star \tau (d_A + \delta_A) + [\ssD_1^\ast,\rho(\gamma)^{-1}] + \ad \Phi,
	\label{E:callias_op}
\end{equation}
where $\tau = -1$ on $\Lambda^0$ and $\tau = 1$ on $\Lambda^2$. The first term
is a twisting (by $\ad P$) of the self-adjoint Dirac operator $\star \tau (d +
\delta)$, which is known as the {\em odd signature operator} and was first
introduced in \cite{atiyah1975spectral}. The inclusion of the second term
$[\ssD_1^\ast,\rho(\gamma)^{-1}]$ (which has order 0) with the
first determine a {\em Dirac-type} operator modelled on the twisted odd
signature operator. Finally, the third term $\ad\Phi \in \Gamma(X; \End(\ad
P\otimes \Lambda^\odd))$ functions as a skew-adjoint potential term, with
constant rank nullspace bundle defined by $\ad P_0 = \bbC \Phi$ in a
neighborhood of $\pa X$. 

\section{Callias-type operators on scattering manifolds} \label{S:callias_review}

We briefly recall the index formula for operators of the form \eqref{E:callias_op}
proved in \cite{kottke2013callias}.  A general {\em Callias-type operator},
\begin{equation}
	P = D + \Psi \in \cB\scDiff^1(X; V),
	\label{E:callias_type}
\end{equation}
on $X$ consists of a Dirac-type operator $D \in \cB\scDiff^1(X; V)$ with
bounded polyhomogeneous coefficients which is modelled on a self-adjoint,
scattering Dirac operator, along with a skew-adjoint potential $\Psi \in
\Gamma(X; \End(V))$ which has a constant rank nullspace bundle $V_0 =
\Null(\Psi \rst_{\pa X}) \to \pa X$ at infinity.  Here $V \to X$ is a module
over the scattering Clifford algebra bundle $\Cl(X)$ whose fiber at $p \in X$
is the Clifford algebra $\Cl(\scT^\ast_p X, g(p))$, and a scattering Dirac
operator is defined to be the composite
\[
	\Gamma(X; V) \stackrel \nabla \to \Gamma(X; \scT^\ast X\otimes V) \stackrel {\cl}{\to} \Gamma(X; V)
\]
of a (Clifford compatible) scattering connection with the Clifford action of
$\scT^\ast X \subset \Cl(X)$ on $V$. A Dirac-type operator differs from this
by a 0th order term, assumed to have order $\cO(x)$ at $\pa X.$

It is assumed that the connection $\nabla$ is the lift of a `true' or `b-'
connection, meaning that $\nabla_v = x\nabla_{\wt v}$ for any vector field $v$
which is bounded with respect to the scattering metric, where $\wt v = x^{-1}
v$ is bounded with respect to the conformally related b-metric $\wt g = x^2 g$.
It follows that $D = x\wt D$ where $\wt D \in \cB\bDiff^1(X; V)$ is a
b-differential operator in the sense of Melrose \cite{melrose1993atiyah}. It is
further assumed that the connection and potential are compatible near infinity,
in the sense that $\nabla \Psi = \cO(x^{1 + \epsilon})$ for some $\epsilon >
0$. 

Under these assumptions, it is shown in \cite{kottke2013callias} that such an
operator \eqref{E:callias_type} admits bounded extensions 
\[
	P : \cH^{\gamma,\gamma+1,k,1}(X; V) \to \cH^{\gamma+1,\gamma+1,k,0}(X; V),
\]
where the Sobolev spaces are defined as in the previous section, with respect
to an extension of the splitting $V \rst_{\pa X} = V_0 \oplus V_1$, where $V_1
= V_0^\perp.$ It is convenient at this point to work with the parameter $\alpha
= \gamma + \tfrac 1 2$, which simplifies the formula
\eqref{E:defect_property_symm} below.

\begin{thm}[\cite{kottke2013callias}]
For $\alpha  = \gamma + \tfrac 1 2 \notin \bspec(\wt D_0)$, the extension 
\[
	P = D + \Psi : \cH^{\alpha-1/2, \alpha+1/2, k,1}(X; V) \to \cH^{\alpha+1/2, \alpha+1/2, k,0}(X; V)
\]
is Fredholm, with index (which is independent of $k$)
\begin{equation}
	\ind(P,\alpha) = \ind(\eth^+_+) + \defect(\wt D_0,\alpha) \in \bbZ.
	\label{E:index_formula}
\end{equation}
Here $\eth^+_+ \in \Diff^1(\pa X; V^+_+, V^-_+)$ is one half of the graded Dirac
operator induced by $D$ on $\pa X$, where $V_+ \subset V \rst_{\pa X}$ is the
positive imaginary eigenbundle of $\Psi \rst_{\pa X}$ and $V^+_+\oplus V^-_+$
denotes the further splitting into positive/negative eigenbundles of $i \cl(x^2\pa_x)$.
Additionally, $\wt D_0 = x^{-(n+1)/2} D_0 x^{(n-1)/2}$, where $n = \dim(X)$ and $D_0$
is a formal expansion at $\pa X$ of the $V_0$ restriction of $D$, and the {\em
defect index} $\defect(\wt D_0,\alpha)$ satisfies
\begin{equation}
	\defect(\wt D_0, \alpha_0 - \epsilon) - \defect(\wt D_0, \alpha_0 + \epsilon) = \dim F(\wt D_0, \alpha_0)
	\label{E:defect_property_diff}
\end{equation}
for $\alpha_0 \in \bspec(\wt D_0)$ and sufficiently small $\epsilon$, where
$F(\wt D_0, \alpha_0)$ is the formal nullspace of $\wt D_0$ at $\alpha_0 \in
\bspec(\wt D_0).$ If in addition $\wt D_0$ (or equivalently $D_0$) is
self-adjoint, then
\begin{equation}
	\defect(\wt D_0, - \alpha) = - \defect(\wt D_0, \alpha).
	\label{E:defect_property_symm}
\end{equation}
\label{T:index_theorem}
\end{thm}

The first term, $\ind(\eth_+^+)$ is well-known from the classical Callias index
theorem in which $\Psi \rst_{\pa X}$ is invertible, see \cite{anghel1993index},
\cite{rade1994callias}, \cite{bunke1995k} and \cite{kottke2011index}. The
second term, $\defect(\wt D_0,\alpha)$, comes from the b-calculus of Melrose
\cite{melrose1993atiyah}. We consider these now in more detail. 

\subsection{Dirac operators near the boundary} \label{S:dirac_near_pa}

Generally speaking, a scattering Dirac operator $D = \sum_{i=0}^{n-1} \cl(e_i)
\nabla_{e_i}$ (where $\set{e_i}$ is an orthonormal frame such that $e_0 = x^2\pa_x$
and $\nabla$ is the lift of a true or b- connection) decomposes near $\pa X$ as
\begin{equation}
	D = \cl(e_0)\bpns{\nabla_{e_0} + \textstyle \sum_{i=1}^{n-1} \cl(e_i\,e_0)\nabla_{e_i}}
	= x\, \cl(e_0)\bpns{\nabla_{\tilde e_0} + \underbrace{\textstyle\sum_{i=1}^{n-1} \cl(e_i\,e_0) \nabla_{\tilde e_i}}_{\eth(x)}}.
	\label{E:dirac_decomp}
\end{equation}
Here $\tilde e_i = x^{-1} e_i$ comprise an orthonormal frame on the {\em
b-tangent bundle} $\bT X$ (see \cite{melrose1993atiyah}) with respect to the
b-metric $\wt g = x^2 g = \tfrac{dx^2}{x^2} + h$; in particular $\set{\tilde
e_i}_{i=1}^{n-1}$ is an orthonormal frame on $\pa X$ with respect to the metric
$h.$ Over $\pa X$, the Clifford module $V$ decomposes as $V \rst_{\pa X} = V^+
\oplus V^-$ into $\pm 1$ eigenspaces for $i \cl(e_0)$, and
\[
\begin{gathered}
	\cl_{\pa} : \Cl(\pa X) \to \End_{\bbZ_2}(V^+\oplus V^-),
	\\ \cl_{\pa}(\tilde e_i) := \cl(e_i\,e_0), \quad 1 \leq i \leq n-1,
\end{gathered}
\]
defines a graded Clifford action of $\Cl(\pa X)$. (Here we use $\Cl(T \pa X, h)
\subset \Cl(\bT X; \wt g)$ along with the isomorphism $\Cl(\bT X, \wt g) \cong
\Cl(\scT X, g)$ defined by multiplication by $x^{-1}$; see
Proposition~\ref{P:sc_LC} below.) It follows that the {\em induced boundary
operator}
\begin{equation}
	\eth(0) = \begin{pmatrix} 0 & \eth^-\\\eth^+ & 0\end{pmatrix} = \sum_{i=1}^{n-1} \cl_{\pa}(\tilde e_i) \nabla_{\tilde e_i}
	\in \Diff^1(\pa X; V^+\oplus V^-)
	\label{E:boundary_dirac}
\end{equation}
is a graded Dirac operator on $\pa X$. (In the case that $D$ is a Dirac-type
operator, there will be additional lower order terms in \eqref{E:dirac_decomp},
though by assumption they are $\cO(x)$ so that $\eth$ is still well-defined as
a Dirac-type operator on $\pa X$.)

For a Callias-type operator, the compatibility condition $\nabla \Psi =
\cO(x^{1+\epsilon})$ implies that 
\[
	D = D_0\oplus D_+\oplus D_- + \cO(x^{1+ \epsilon}),
\]
with respect to an extension of the splitting $V \rst_{\pa X} = V_0 \oplus V_+
\oplus V_-$ into the nullspace and positive/negative imaginary eigenspaces of
$\Psi \rst_{\pa X}$. It follows that \eqref{E:dirac_decomp} and
\eqref{E:boundary_dirac} apply separately to $D_0$, $D_+$ and $D_-$, these being
the $\R_+ = (0,\infty)$ invariant operators on $\pa X \times \R_+$
obtained by freezing the coefficients of $D$ at the boundary and projecting to
$V_0$, $V_+$ or $V_-$, respectively.  

The conclusions of Theorem~\ref{T:index_theorem} refer in particular to the
induced operator $\eth^+_+ \in \Diff^1(\pa X; V^+_+,V^-_+)$ of $D_+$, and to
$\wt D_0 = x^{-(n+1)/2} D_0 x^{(n-1)/2}$, which should be understood as a
conjugation of $D_0$ by $x^{n/2}$ and a factoring out of $x^{1/2}$ from the
left and right. (In particular $D_0$ is formally self-adjoint with respect to
the metric $g$ on $X$ if and only if $\wt D_0$ is formally self-adjoint with
respect to $\wt g = x^2 g.$)

Explicitly, if we take $V$ in radial gauge with respect to $\nabla$, so that
$\nabla_{\tilde e_0} \equiv x\pa_x$, we may write \eqref{E:dirac_decomp} in
local coordinates $(x,y_1,\ldots,y_{n-1})$ as
\[
\begin{aligned}
	D &= x\,a(x,y)\big(x\pa_x + \textstyle \sum_{i=1}^n b_i(x,y)\pa_{y_i} + c(x,y)\big),
	\\ D_0 &= \Pi_0\, x\,a(0,y)\big(x\pa_x + \textstyle \sum_{i=1}^n b_i(0,y)\pa_{y_i} + c(0,y)\big)\Pi_0,
	\\ \wt D_0 &= \Pi_0\, a(0,y)\big(x\pa_x + \tfrac{n-1}{2} + \textstyle \sum_{i=1}^n b_i(0,y)\pa_{y_i} + c(0,y)\big)\Pi_0,
	\\ &= \Pi_0\, \cl(e_0)\big(x\pa_x + \tfrac{n-1}{2} + \eth \big)\Pi_0.
\end{aligned}
\]
(Note that only $x\pa_x$ fails to commute with $x^{(n-1)/2}$, and
$[x\pa_x,x^{(n-1)/2}] = \tfrac{n-1}{2}$.) The discrete set of {\em indicial roots},
$\bspec(\wt D_0) \subset \bbR$, consists of those $\alpha \in \bbR$ for which
the Mellin transformed operator
\[
	I(\wt D_0,\alpha) = \Pi_0\, \cl(e_0)\big(\alpha + \tfrac{n-1}{2} + \eth\big)\Pi_0,
\]
is not invertible, and then $F(\wt D_0,\alpha_0) \subset C^\infty(\pa X; V_0)$
is the (necessarily finite-dimensional) nullspace of $I(\wt D_0, \alpha_0).$ In
fact, the defect index is just the formal index of $\wt D_0$, and the
properties \eqref{E:defect_property_diff} and \eqref{E:defect_property_symm}
follow from the {\em relative index theorem} in \cite{melrose1993atiyah}.

\section{Index of the deformation complex} \label{S:comp}

We return now to the consideration of \eqref{E:callias_op}, first verifying
that it satisfies the necessary conditions to apply
Theorem~\ref{T:index_theorem}. Here $V = \ad P \otimes \Lambda^\odd$, and the
connection defining the Dirac operator is $\nabla = d_A \otimes
\nabla^{\LC(g)}$.  Since $A$ is a true connection by assumption, the fact that
$\nabla$ is the lift of a b-connection follows from the next result, which is
of independent interest. 

\begin{prop}
The Levi-Civita connection on a scattering manifold of dimension $n$ with
metric $g = \tfrac{dx^2}{x^4} + \tfrac{h}{x^2}$ is a lift of a b-connection. In
fact, multiplication by $x^{-1}$ induces an isomorphism of $\scT X$ and the
b-tangent bundle $\bT X$ and of their associated principal frame bundles,
identifying $g$ with $\wt g = x^2 g = \tfrac{dx^2}{x^2} + h$.
In terms of this isomorphism,
\begin{equation}
	\nabla^{\LC(g)} \cong \nabla^{\LC(\wt g)} + B,
	\quad B = \sum_{i=1}^{n-1}E_{0i} \tilde e_i',
	\label{E:LC_connections}
\end{equation}
where $\set{\tilde e'_i} \subset \bT^\ast X$ is the dual to an orthonormal frame $\set{\tilde
e_0 = x\pa_x,\tilde e_1,\ldots,\tilde e_{n-1}}$ for $\bT X$ and $E_{0i} \in \so(n)$ acts by
$E_{0i}\tilde e_i = \tilde e_0$, $E_{0i} \tilde e_0 = - \tilde e_i$ and is $0$ otherwise.
\label{P:sc_LC}
\end{prop}
\noindent The meaning of \eqref{E:LC_connections} is that if $v$ is a
scattering vector field, equal to $x\,\wt v$ for a b-vector field $\wt v$, then
$\nabla^{\LC(g)}_v = x\big(\nabla^{\LC(\wt g)}_{\wt v} + B(\wt v)\big)$.
\begin{proof}
Let $\set{e_0 = x^2\pa_x, e_1 = x\tilde e_1,\ldots, e_{n-1} = x\tilde e_{n-1}}$ be
the orthonormal frame for $\scT X$ which is identified with $\set{\tilde e_i}$ by
the isomorphism. The Koszul formula for $g$ along with the fact that
$[e_0,e_j] = xe_j$, $j \geq 1$, implies
\[
	\nabla^{\LC(g)}_{e_j}e_0 = -x e_j, 
	\quad \nabla^{\LC(g)}_{e_j}e_k = x(e_0\delta_{jk} + \nabla^{\LC(h)}_{\tilde e_j}\tilde e_k)
\]
for $j,k \geq 1.$ On the other hand, from the Koszul formula for $\wt g$ it follows that
\[
	\nabla^{\LC(\wt g)}_{\tilde e_j} \tilde e_0 = 0,
	\quad \nabla^{\LC(\wt g)}_{\tilde e_j} \tilde e_k = \nabla^{\LC(h)}_{\tilde e_j} \tilde e_k.
\]
Comparing these formulas leads immediately to \eqref{E:LC_connections}.
\end{proof}

In \eqref{E:callias_op} $\ad \Phi$ plays the role of the potential term $\Psi$,
and the nullspace bundle is simply $V_0 = \ad P_0 \otimes \Lambda^\odd \cong
\Lambda^\odd$. The compatibility of the connection and the potential follows
from finiteness of the action \eqref{E:action}:
\[
	\nabla \Psi = d_A \Phi \in L^2(X; \ad P \otimes \Lambda^1)
	\implies d_A \Phi = \cO(x^{3/2+\epsilon}).
\]

The Clifford action is best understood as follows. First, we make use of
the vector bundle isomorphism $\Lambda^\ast X \cong \Cl(X)$ to simplify
computations. This isomorphism intertwines the Hodge star with the {\em
normalized Clifford volume element} $\omega_\C \in \Cl(X)$ up to a sign:
\begin{equation}
\begin{gathered}
	\End(\Lambda^\ast X) \ni \star \tau \cong \omega_\C \in \Cl(X)
	\\ \tau := i^{\left[\frac{n+1}{2}\right] + k(k-1) + 2nk}\ \text{on $\Lambda^k$}, 
	\quad \omega_\C := i^{\brackets{\frac{n+1}{2}}} e_0\cdots e_{n-1},
	\quad n = \dim(X).
\end{gathered}
	\label{E:clifford_hodge}
\end{equation}
Here $\set{e_i}$ is any orthonormal frame, and $\tau$ is the general version of
the sign operator appearing in \eqref{E:callias_op}. Note that in the case $n =
2l$ is even, $\tau = i^{k(k-1) + l}$ and the $\pm 1$ eigenspaces of $\omega_\C
= \star \tau$ define the {\em signature splitting} $\Lambda^\ast X = \Lambda^+X
\oplus \Lambda^- X$. 

On an odd-dimensional manifold, the odd signature operator is the Dirac
operator on odd forms associated to the Levi-Civita connection and the {\em odd
Clifford action}:
\[
\begin{gathered}
	\star \tau (d + \delta) = \sum_i \cl_\odd(e_i) \nabla^{\LC}_{e_i}  \in \Diff^1(X; \Lambda^\odd),
	\\ \cl_\odd : \Cl(X) \to \End(\Lambda^\odd) \cong \End(\Cl^1(X)),
	\quad \cl_\odd(e) := \omega_\C e \cdot
\end{gathered}
\]
The first term in \eqref{E:callias_op} is the twisting of this operator by $\ad
P$ via the connection $A$.

Finally, note that the term $[\ssD_1^\ast,\rho(\gamma)^{-1}]$ in
\eqref{E:callias_op} only involves commutators of $x^2\pa_x$ with powers
$x^{-2\gamma}$ and $x^{-2\gamma - 1}$, and these commutators have order
$\cO(x)$ near $\pa X$. The observations of this section together prove:
 
\begin{prop}
The operator \eqref{E:callias_op} is a Callias-type operator in the sense
of \cite{kottke2013callias}.
\label{P:callias_is_callias}
\end{prop}

\subsection{Induced operators and indicial roots} \label{S:induced_and_indicial}

It remains to determine the induced operator $\eth_+^+$ as well as $\wt D_0$ and
its indicial roots. To apply the considerations of \S\ref{S:dirac_near_pa} to
the operator \eqref{E:callias_op}, we first identify $\Lambda^\odd X$ with
$\Lambda^\ast \pa X$ near $\pa X$ via
\begin{equation}
	\Lambda^\odd X \cong \Cl^1(X) \ni \left\{ \begin{array}{rcrc}
	e_0 e_I &\leftrightarrow  &-\tilde e_I \in \Cl^0(\pa X) \cong \Lambda^\even\pa X, 
		& \abs I \ \text{even,} \\
	e_J &\leftrightarrow &\tilde e_J \in \Cl^1(\pa X) \cong \Lambda^\odd\pa X, 
		& \abs J\ \text{odd,} \end{array}\right.
	\label{E:odd_to_all}
\end{equation}
where $I$ and $J$ are multi-indices: $e_I = e_{i_1}\cdots e_{i_m}$ and $\abs I =
m.$
\begin{lem}
Under the identification \eqref{E:odd_to_all}, $\cl_\odd(e_0) \cong -i (\star
\tau)_{\pa X}$; in particular $i \cl(e_0)$ generates the signature splitting 
\[
	\Lambda^\ast \pa X = \Lambda^+\pa X\oplus \Lambda^-\pa X.
\]
The induced Clifford action $\cl_\pa : \Cl(\pa X) \to \End_{\bbZ_2}(\Lambda^+\pa X\oplus
\Lambda^-\pa X)$ associated to $\cl_\odd$ is the standard Clifford action on forms.
\label{L:induced_clifford}
\end{lem}
\begin{proof}
The Clifford volume element defined in \eqref{E:clifford_hodge} may be
expressed as $\omega_\C = ie_0 \omega'_\C,$ where $\omega'_\C$ is the volume
element for $\Cl(\pa X)$. Thus
\[
	i \cl_\odd(e_0) = i \omega_C e_0 = - e_0\omega'_\C e_0 = \omega'_\C \cong (\star \tau)_{\pa X}
\]
which generates the signature splitting on the even dimensional manifold $\pa
X$ as remarked above. Likewise, recalling that $\omega_\C$ is an involution
which is central in odd dimensions, so that $\cl_\odd(e_je_0) =
\omega_\C e_j \omega_\C e_0 = e_j e_0$, the induced action is given by
\[
\begin{aligned}
	\cl_\pa(\tilde e_j) \tilde e_I &\cong (e_j e_0)( - e_0 e_I) = e_j e_I \cong \tilde e_j \tilde e_I, \\
	\cl_\pa(\tilde e_j) \tilde e_J &\cong (e_j e_0) e_J = -e_0 e_j e_J \cong \tilde e_j \tilde e_J,
\end{aligned}
\]
for $\abs I$ even and $\abs J$ odd.
\end{proof}

It is convenient to take $\nabla = d_A \otimes \nabla^{\LC(g)}$ to be in radial
gauge, so that $\nabla_{x^2\pa_x} = x^2\pa_x.$ The condition $d_A \Phi
\rst_{\pa X} = 0$ implies that $A$ restricts separately to a connection on each
of the summands $\ad P_0$, $\ad P_+$ and $\ad P_-$ over $\pa X$, and
Proposition~\ref{P:sc_LC} implies that $\nabla^{\LC(g)}$ restricts to the
connection $\nabla^{\LC(h)} + B$ on forms over $\pa X$. 

In light of Lemma~\ref{L:induced_clifford}, it follows that induced Dirac
operators $\eth_\pm$ coincide, modulo lower order terms, with the (even)
signature operator $d + \delta$ on $\pa X$, twisted by $\ad P_\pm$. Since only
the index of $\eth^+_+$ appears in Theorem~\ref{T:index_theorem}, the lower
order terms may be ignored, and invoking Lemma~\ref{L:monopole_bundle_isoms} we
therefore have:

\begin{prop}
For the operator \eqref{E:callias_op}, the induced operator $\eth_+^+$ is
homotopic to the twisted signature operator
\[
	\eth_+^+ \sim (d_A + \delta_A)^+  \in \Diff^1(\pa X; \Lambda^+\pa X\otimes L^2, \Lambda^-\pa X\otimes L^2),
\]
where $L \to \pa X$ is the line bundle of degree $k$ defining the charge,
equipped with the connection induced by $A$.
\label{P:boundary_dirac}
\end{prop}

When considering $\wt D_0$, the lower order terms are of critical importance,
as they affect the locations of the indicial roots.

\begin{prop}
For the operator \eqref{E:callias_op}, the operator $\wt D_0$ is given by
\begin{equation}
	\wt D_0 = -i (\star \tau)_{\pa X} (x\pa_x + (d + \delta)_{\pa X} + N) \in \Diff^1(\pa X\times \R_+; \Lambda^\ast \pa X),
	\label{E:D_0}
\end{equation}
where $N = -1-2\gamma$ on $\Lambda^0\pa X$, $N = 0$ on $\Lambda^1\pa X$, and $N = 1$ on $\Lambda^2\pa X$.
\label{P:D_0}
\end{prop}
\begin{proof}
The bundle $\ad P_0 \to \pa X$ is explicitly trivialized by $\Phi$, and it
follows from the discussion in \S\ref{S:framing} that the induced connection on
it is not only flat, but in fact trivial. Thus the twisting by $\ad P_0$ may be
disregarded completely. Then following the discussion in
\S\ref{S:dirac_near_pa} and using Proposition~\ref{P:sc_LC}, 
\[
	D_0 = x\cl(e_0)\big(x\pa_x + \textstyle\sum_{i \geq 1}\cl_\pa(\tilde e_i) (\nabla^{\LC(h)}_{\tilde e_i} + B(\tilde e_i))\big) + [\ssD_1^\ast,\rho(\gamma)^{-1}].
\]
As already remarked, $\cl(e_0) = -i(\star \tau)_{\pa X}$ and $\sum_{i\geq 1}
\cl_\pa(\tilde e_i) \nabla^{\LC(h)}_{\tilde e_i} = (d + \delta)_{\pa X}$, so it
remains to determine the contribution from the last two terms.

The first of these is $\cl_\pa(\tilde e_i) B(\tilde e_i) = \cl_\pa(\tilde e_i) E_{0i}$.
The endomorphism $E_{0i}$ of $\scT X$ in \eqref{E:LC_connections} is
represented by the same matrix in the contragredient representation (i.e.\ on
$\scT^\ast X$) by skew-adjointness, and acts on $\Lambda^\ast X \cong \Cl(X)$
as an (ungraded) derivation. Thus 
\[
	E_{0i} e_J = e_{J(i,0)}, \quad E_{0i} e_0 e_I = -e_i e_I + e_0 e_{I(i,0)}
\]
where $e_{J(i,0)}$ is the element obtained by replacing $e_i$ by $e_0$ in $e_J$
if it occurs and which is 0 otherwise. Then
\[
\begin{gathered}
	\cl_\pa(\tilde e_i) E_{0i} \tilde e_J  \cong e_i e_0 E_{0i} e_J = e_i e_0 e_{J(i,0)}
	\cong \begin{cases} -\tilde e_J & i \in J,
	\\ 0 & i \notin J, \end{cases}
	\\ \cl_\pa(\tilde e_i) E_{0i} \tilde e_I \cong e_i e_0 E_{0i} (-e_0 e_I) = e_i e_0(e_i e_I - e_0 e_{I(i,0)})
	\cong \begin{cases} 0 & i \in I,
	\\ -e_I & i \notin I. \end{cases}
\end{gathered}
\]
Thus $\sum_i\cl_\pa(\tilde e_i) B(\tilde e_i)$ acts by $-k$ on $\Lambda^k\pa X$
for $k$ odd, and by $-(m - k)$ for $k$ even, where $m = \dim(\pa X) = 2.$

The final term to consider is $[\ssD_1^\ast,\rho(\gamma)^{-1}].$ Since we only
consider the part of the operator acting on $\ad P_0$, we can replace
$\rho(\gamma)^{-1}$ by $x^{-2\gamma}$, and as noted above ignore the twisting
and consider only the action on forms. From \eqref{E:def_complex}, we see that
$\ssD_1^\ast$ has order 0 on $\Lambda^3 X$, so this will not contribute to
the commutator.  Thus we may restrict attention to the part of $\ssD_1^\ast =
\star \tau \delta = \sum_i\cl_\odd(e_i) \nabla^{\LC(g)}_{e_i}$ mapping sections
of $\Lambda^1 X$ to sections of $\Lambda^3 X$. 

Only the $\nabla^{\LC(g)}_{e_0} = e_0 = x^2\pa_x$ term will contribute to the
commutator (since $e_j$, $j \neq 0$ can be chosen to commute with $x$), and the
only 1-forms mapped by $\cl_\odd(e_0) = \omega_\C e_0$ into $\Lambda^3X$ are
those proportional to $e_0$; indeed $\cl_\odd(e_0)$ sends $e_1$ and $e_2$ into
$\Lambda^1 X.$ Thus since $[x^2\pa_x,x^{-2\gamma}] = x(-2\gamma)$, it follows that
\[
	[\ssD_1^\ast,x^{-2\gamma}] = x\cl_\odd(e_0)(-2\gamma)\rst_{\sspan(e_0)}.
\]
Since $\sspan(e_0) \subset \Lambda^1 X$ is identified with $\Lambda^0 \pa X$ by
the isomorphism \eqref{E:odd_to_all}, the net effect of
$[\ssD_1,\rho(\gamma)^{-1}]$ is multiplication by $-2\gamma$ on $\Lambda^0\pa
X$. Thus
\[
	D_0 = x\cl(e_0)\big(x\pa_x + (d + \delta)_{\pa X} + M), \quad M 
	= \begin{cases} -2 - 2\gamma & \text{on $\Lambda^0 \pa X$,}
	\\ -1 &\text{on $\Lambda^1 \pa X$,}
	\\ 0 &\text{on $\Lambda^2 \pa X$.} \end{cases}
\]
Taking $\wt D_0 = x^{-(n+1)/2} D_0 x^{(n-1)/2}$ has the effect of removing the
overall factor of $x$ and adding $\tfrac{n-1}{2} = 1$ to all terms, so \eqref{E:D_0}
follows.
\end{proof}

\begin{prop}
The indicial roots of $\wt D_0$ are
\begin{equation}
\begin{gathered}
	\bspec(\wt D_0) =
	\set{-\tfrac 1 2 \pm \sqrt{\nu + \tfrac 1 4}}
	\cup \set{\tfrac{1 + 2\gamma} 2 \pm \sqrt{\nu + \tfrac {(1 + 2\gamma)^2} 4}}, \quad \nu \in \spec(\Delta_{\pa X}).
\end{gathered}
	\label{E:ind_roots}
\end{equation}
The formal nullspaces associated to the roots $\set{-1,0,1 + 2\gamma}$
(for whic $\nu = 0$) are the harmonic forms of degree $2$, $1$, and $0$
respectively:
\[
	F(\wt D_0,-1) \cong H^2(\pa X; \bbR), 
	\quad F(\wt D_0,0) \cong H^1(\pa X; \bbR), 
	\quad F(\wt D_0,1 + 2\gamma) \cong H^0(\pa X; \bbR).
\]
\label{P:ind_roots}
\end{prop}
\noindent Technically speaking, we should distinguish between the contributions
to $\bspec(\wt D_0)$ coming from eigenvalues of $\Delta_{\pa X}$ acting on
$\Lambda^k \pa X$ for various $k$; however since $\dim(\pa X) = 2$, the spectrum of
$\Delta_{\pa X}$ is the same on forms of any degree.
\begin{proof}
The term $\cl(e_0) = -i (\star \tau)_{\pa X}$ in \eqref{E:D_0} is a bundle
isomorphism and may be ignored. Taking the Mellin transform replaces $x\pa_x$
by $\lambda$; therefore we consider the invertibility of 
\begin{equation}
	\begin{pmatrix} \lambda -1 - 2\gamma & \delta & 0
	\\ d & \lambda & \delta
	\\ 0 & d & \lambda + 1 \end{pmatrix}
	\label{E:ind_root_op}
\end{equation}
on $\pa X$, with respect to $\Lambda^0\pa X\oplus \Lambda^1\pa X\oplus \Lambda^2\pa X.$ On
the harmonic forms, this is degenerate for $\lambda \in \set{-1, 0, 1 +
2\gamma}$ with nullspace consisting of harmonic forms of the associated degree,
giving $F(\wt D_0,\lambda)$ as claimed.

Off of the harmonic forms, we use the fact that the only coupling is between
closed and coclosed forms of relative degree 1. Thus it suffices to consider
invertibility on pairs $(\phi_\nu,\psi_\nu) \in C^\infty(\pa X:
\Lambda^k)\oplus C^\infty(\pa X; \Lambda^k)$ such that $d\phi_\nu = \sqrt\nu
\psi_\nu$ and $\delta \psi_\nu = \sqrt \nu \phi_\nu$ for $k = 0$ or $k = 1$, on
which \eqref{E:ind_root_op} takes the form
\[
	\begin{pmatrix} \lambda - 1 - 2\gamma & \sqrt \nu
	\\ \sqrt \nu & \lambda\end{pmatrix} 
	\quad \text{or}\quad
	\begin{pmatrix} \lambda & \sqrt \nu
	\\ \sqrt \nu& \lambda + 1\end{pmatrix},
\]
respectively. These give the right and left hand contributions to
\eqref{E:ind_roots} for $\nu > 0.$
\end{proof}

\subsection{The virtual dimension} \label{S:comp_index} 

It is convenient to divide the indicial roots \eqref{E:ind_roots} into the
`geometric' roots, with $\nu > 0$, and the `topological' roots
$\set{-1,0,1 + 2\gamma}$ for which $\nu = 0.$ The former are sensitive to the
metric $h$ on $\pa X$ and in particular may be scaled away from $0$ by altering
$g$.  On the other hand, the topological roots are independent of the metric. (This division
of indicial roots is well-known; see for instance \cite{albin2012signature}.)
These sets may be further subdivided into `variable' roots, which depend on
$\gamma$, and and `static' roots, which do not. These are depicted in
Figure~\ref{F:bspec}, with static roots represented by solid dots, variable
ones by hollow dots, and with the topological roots drawn larger than the
geometric ones; the parameter $\alpha = \gamma + \tfrac 1 2$ appearing in
Theorem~\ref{T:index_theorem} is also plotted. The static geometric roots are
symmetric about $-\tfrac{1}{2}$, and always bounded away from it by at least $\tfrac 1 2$. The
variable geometric roots are symmetric about $\alpha$.
Consider the following regimes:
\begin{itemize}
\item ($\gamma = 0$): $\alpha = \tfrac 1 2$ and the b-spectrum is symmetric since here $\wt D_0$ 
is formally self-adjoint.
\item ($-\tfrac 1 2 < \gamma < 0$): $\alpha$ lies above the static topological root $0$ and below the lone 
variable topological root $1 + 2\gamma$. There may also be static geometric
roots in this range, but for $\gamma$ sufficiently close to $- \tfrac 1 2$ there are no
roots between $\alpha$ and $0$.
\item ($\gamma = -\tfrac 1 2$): $\alpha$, the variable topological root, and the static topological root at $0$ coincide.
\item ($\gamma < -\tfrac 1 2$): $\alpha$ lies above the variable topological root $1
+ 2\gamma$ and below the static root $0$. For $\gamma$ sufficiently close to
$-\tfrac 1 2$, there are no geometric roots (either static or variable) between
$\alpha$ and $0$.
\end{itemize} 

\begin{figure}[tb]
\begin{center}
\begin{tikzpicture}
\matrix (m) [matrix of nodes, row sep=1cm]
{
\begin{tikzpicture}[scale=2]
	\draw[<->,thick] (-2.3,0) -- (2.3,0);
	\draw (0,0.10) -- (0,-0.10) node [below] {$0$};
	\draw (1,0.10) -- (1,-0.10) node [below] {$1$};
	\draw (-1,0.10) -- (-1,-0.10) node [below] {$-1$};
	\draw (2,0.10) -- (2,-0.10) node [below] {$2$};
	\draw (-2,0.10) -- (-2,-0.10) node [below] {$-2$};

	\fill (0,0) circle (1.5pt);
	\fill (-1,0) circle (1.5pt);

	\fill (0.30,0) circle (1.0pt);
	\fill (0.77,0) circle (1.0pt);
	\fill (1.5,0) circle (1.0pt);
	\fill (-1.35,0) circle (1.0pt);
	\fill (-1.77,0) circle (1.0pt);

	\draw[thick] (1,0) circle (1.5pt);

	\draw (-0.30,0) circle (1.0pt);
	\draw (-0.77,0) circle (1.0pt);
	\draw (-1.5,0) circle (1.0pt);
	\draw (1.35,0) circle (1.0pt);
	\draw (1.77,0) circle (1.0pt);

	\draw[thick] (0.5,0.15) node[above] {$\alpha$} -- (0.5,-0.15); 

	\path (0,0.10) node[above] {$b^1(\pa X)$};
	\path (-1,0.10) node[above] {$b^2(\pa X)$};
	\path (1,0.10) node[above] {$b^0(\pa X)$};

	\path (-2.5,0.2) node {a)};
\end{tikzpicture} 
\\
\begin{tikzpicture}[scale=2]
	\draw[<->,thick] (-2.3,0) -- (2.3,0);
	\draw (0,0.10) -- (0,-0.10) node [below] {$0$};
	\draw (1,0.10) -- (1,-0.10) node [below] {$1$};
	\draw (-1,0.10) -- (-1,-0.10) node [below] {$-1$};
	\draw (2,0.10) -- (2,-0.10) node [below] {$2$};
	\draw (-2,0.10) -- (-2,-0.10) node [below] {$-2$};

	\fill (0,0) circle (1.5pt);
	\fill (-1,0) circle (1.5pt);

	\fill (0.30,0) circle (1.0pt);
	\fill (0.77,0) circle (1.0pt);
	\fill (1.5,0) circle (1.0pt);
	\fill (-1.35,0) circle (1.0pt);
	\fill (-1.77,0) circle (1.0pt);

	\draw[thick] (0.4,0) circle (1.5pt);

	\draw (-0.15,0) circle (1.0pt);
	\draw (-0.47,0) circle (1.0pt);
	\draw (-0.8,0) circle (1.0pt);
	\draw (-1.5,0) circle (1.0pt);
	\draw (0.55,0) circle (1.0pt);
	\draw (0.87,0) circle (1.0pt);
	\draw (1.2,0) circle (1.0pt);
	\draw (1.9,0) circle (1.0pt);

	\draw[thick] (0.2,0.15) node[above] {$\alpha$} -- (0.2,-0.15); 


	\path (-2.5,0.2) node {b)};
\end{tikzpicture} 
\\
\begin{tikzpicture}[scale=2]
	\draw[<->,thick] (-2.3,0) -- (2.3,0);
	\draw (0,0.10) -- (0,-0.10) node [below] {$0$};
	\draw (1,0.10) -- (1,-0.10) node [below] {$1$};
	\draw (-1,0.10) -- (-1,-0.10) node [below] {$-1$};
	\draw (2,0.10) -- (2,-0.10) node [below] {$2$};
	\draw (-2,0.10) -- (-2,-0.10) node [below] {$-2$};

	\fill (0,0) circle (1.5pt);
	\fill (-1,0) circle (1.5pt);

	\fill (0.30,0) circle (1.0pt);
	\fill (0.77,0) circle (1.0pt);
	\fill (1.5,0) circle (1.0pt);
	\fill (-1.35,0) circle (1.0pt);
	\fill (-1.77,0) circle (1.0pt);

	\draw[thick] (-0.8,0) circle (1.5pt);

	\draw (-1.20,0) circle (1.0pt);
	\draw (-1.67,0) circle (1.0pt);
	\draw (0.4,0) circle (1.0pt);
	\draw (0.87,0) circle (1.0pt);
	\draw (1.6,0) circle (1.0pt);

	\draw[thick] (-0.4,0.15) node[above] {$\alpha$} -- (-0.4,-0.15); 


	\path (-2.5,0.2) node {c)};
\end{tikzpicture}
\\ };
\end{tikzpicture} \caption{The b-spectrum of $\wt D_0$. Static roots are solid, variable roots are hollow, and topological
roots are depicted as larger than geometric roots. (a) $\gamma = 0
\iff \alpha = \tfrac 1 2$. (b) $\gamma \in (-\tfrac 1 2,-\tfrac 1 2 + \lambda_1) \iff \alpha \in
(0,\lambda_1)$. (c) $\gamma \in (-\tfrac 3 2,-\tfrac 1 2) \iff \alpha \in (-1,0).$}
\label{F:bspec}
\end{center}
\end{figure}
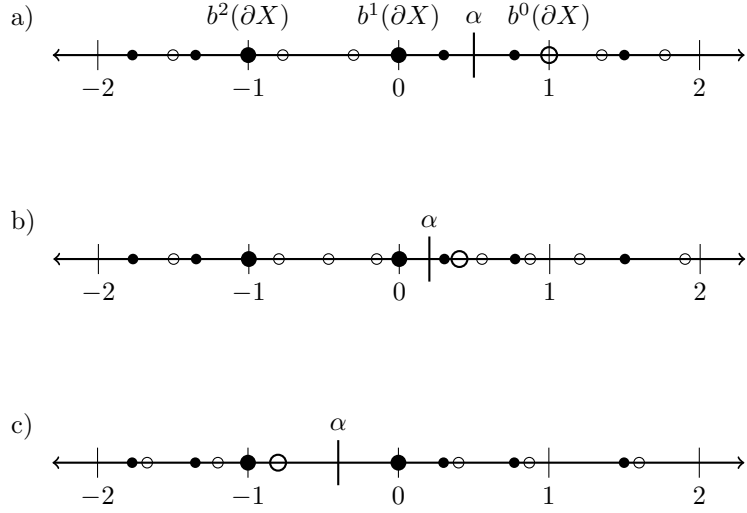

\begin{thm}
The monopole deformation complex \eqref{E:Hilb_complex} is Fredholm for $\gamma
\in (- \tfrac 1 2,- \tfrac 1 2 + \lambda_1)$ and for $\gamma \in (- \tfrac 3 2,- \tfrac 1 2),$ where
$\lambda_1 = \sqrt{\nu_1 + \tfrac 1 4} - \tfrac 1 2$ and $\nu_1$ is the
smallest positive eigenvalue of $\Delta_{\pa X}$. The index, and therefore
virtual dimension, is given by
\[
\begin{aligned}
	\vdim(\frM_k) &= \ind\big(\ssD_2(\gamma) + \ssD_1'(\gamma)\big) = 4\ul k - \tfrac 1 2 b^1(\pa X), 
		& \gamma &\in (-\tfrac 1 2,-\tfrac 1 2+\lambda_1),
	\\ \vdim(\cM_k) &= \ind\big(\ssD_2(\gamma) + \ssD_1'(\gamma)\big) = 4\ul k + \tfrac 1 2 b^1(\pa X) - b^0(\pa X), 
		& \gamma &\in (-\tfrac 3 2,-\tfrac 1 2),
\end{aligned}
\]
where $\ul k = \int_{\pa X}c_1(L) = k_1 + \cdots + k_{b^0(\pa X)}$ is a sum
over components of $\pa X$, and $b^i(\pa X)$ denotes the $i$th Betti number of
$\pa X.$ 
\label{T:vdim}
\end{thm}
\begin{proof}
Combining Proposition~\ref{P:boundary_dirac} with the standard index formula
\cite{lawson1989spin}, Thm.\ 13.9, 
\[
	\ind(\eth^+_+) = \ind\bpns{(d + \delta)^+_{L^2}} = \int_{\pa X} \mathrm{ch}_2(L^2) \mathbf{\hat L}(\pa X)
	= \int_{\pa X} 4c_1(L) = 4\ul k.
\]
Here $\mathrm{ch}_2(E) = \sum_k 2^k \mathrm{ch}^k(E)$ and $\mathrm{ch}^k(E)$
denotes the $H^{2k}(\pa X; \R)$ component of the Chern character
$\mathrm{ch}(E).$ 

The term $\defect(\wt D,\alpha)$ may be computed using
\eqref{E:defect_property_diff} and \eqref{E:defect_property_symm}, though the
second of these identities is only valid when $\wt D_0$ is self-adjoint, which
occurs here exactly when $\gamma = 0$. For this value then, $\alpha = \tfrac 1 2$ and 
\[
	\defect(\wt D_0, \tfrac 1 2) = - \tfrac 1 2 b^1(\pa X) - \textstyle\sum_j F_j
\]
where the sum is over the dimensions $F_j = \dim F(\wt D_0,\lambda_j)$ of the
finitely many (static) geometric indicial roots such that $0 < \lambda_j < \tfrac 1 2$
(see Figure~\ref{F:bspec}.(a)).

As $\gamma$ varies from $0$ toward $-\tfrac 1 2$, $\alpha$ varies from $\tfrac 1 2$ toward
$0$, and each time $\alpha$ passes over a (necessarily static geometric) root
$\lambda_j$, the defect index increases by $F_j$ by
\eqref{E:defect_property_diff}. Once $0 < \alpha < \lambda_1$, where
$\lambda_1 = -\tfrac 1 2 + \sqrt{\nu_1 + \tfrac 1 4}$ is the smallest positive
root, we obtain
\[
	\defect(\wt D_0, \alpha) = - \tfrac 1 2 b^1(\pa X), \quad 0 < \alpha < \lambda_1.
\]
(See Figure~\ref{F:bspec}.(b).) This corresponds precisely to the range $\gamma
\in (-\tfrac 1 2, - \tfrac 1 2 + \lambda_1)$, as claimed.

Finally, as $\gamma$ passes through $-\tfrac 1 2$ from above, $\alpha$ passes over the
static topological root $0$ from above, while at the same time passing over the
variable topological root $1 + 2\gamma$ from below (see
Figure~\ref{F:bspec}.(c)). After this transition, it follows from
\eqref{E:defect_property_diff} that
\[
	\defect(\wt D_0,\alpha) = \tfrac 1 2 b^1(\pa X) - b^0(\pa X), \quad -1 < \alpha < 0.
\]
Indeed, from this point onward the only other roots crossed as $\alpha$
continues to decrease are static ones (since the variable topological root $1 +
2\gamma < \alpha$ from now on and the variable geometric roots are symmetric
about $\alpha$ and bounded away from it by $\sqrt{\nu_1}$), the next being at
$\alpha = -1$, or $\gamma = -\tfrac 3 2.$
\end{proof}

\appendix
\bibliographystyle{amsalpha}
\bibliography{references}

\end{document}